\documentclass[11pt,reqno]{amsart}

\usepackage{amsmath,amsthm, underscore, amssymb}
\usepackage{enumerate}

\newtheorem{theorem}{Theorem}[section]
\theoremstyle{plain}
\newtheorem{corollary}[theorem]{Corollary}
\newtheorem{example}[theorem]{Example}
\newtheorem{lemma}[theorem]{Lemma}
\newtheorem{proposition}[theorem]{Proposition}
\newtheorem{remark}[theorem]{Remark}

\numberwithin{equation}{section}

\newcommand{\CC}{\mathbb{C}}
\newcommand{\NN}{\mathbb{N}}

\newcommand{\jbar}{\bar{j}}
\newcommand{\kbar}{\bar{k}}

\newcommand{\zbar}{\bar{z}}

\newcommand{\fbar}{\bar{f}}
\newcommand{\wbar}{\bar{w}}
\renewcommand{\d}{\partial}

\DeclareMathOperator{\trace}{trace}

\def\hbar{\bar{h}}

\def\Re{\mbox{\rm Re}\,}

\def\betabar{\bar{\beta}}
\def\gamabar{\bar{\gamma}}

\def\npbar{\overline{n+1}}

\begin{document}
\title[The sharp upper bounds for the first positive eigenvalue]{The sharp upper bounds for the first positive eigenvalue of the Kohn-Laplacian on compact strictly pseudoconvex hypersurfaces}
\author{Song-Ying Li}
\address{Department of Mathematics, University of California, Irvine, CA 92697-3875}
\email{sli@math.uci.edu}
\author{Guijuan Lin}
\address{College of Mathematics and Informatics, Fujian Normal University, Fuzhou 350108, Fujian, China}
\email{Guijuan_Lin@163.com}
\author{Duong Ngoc Son}
\address{Texas A\&M University at Qatar, Science Program, PO Box 23874, Education City, Doha, Qatar}
\email{son.duong@qatar.tamu.edu}
\thanks{The second author was partially supported by the Hu Guozan Study-Abroad Grant for graduates (China) for her visit to UC Irvine in 2015--2016 when part 
of this work was done. The third author was partially supported by the Qatar National Research Fund, NPRP project 7-511-1-098. Part of this work was done while the third author visited Fujian Normal University at Fuzhou, China in July 2016 which he thanks for supports and hospitality.}

\date{\today}

\begin{abstract}
We give sharp and explicit upper bounds for the first positive eigenvalue $\lambda_1(\Box_b)$ of the 
Kohn-Laplacian on compact strictly pseudoconvex hypersurfaces in $\mathbb{C}^{n+1}$ in terms of their
defining functions. As an application, we show that in the family of real ellipsoids, $\lambda_1(\Box_b)$ has a unique maximum value at the CR sphere.
\end{abstract}

\maketitle

\section{Introduction}
 Let $(M^{2n+1},\theta)$ be a compact strictly pseudoconvex
pseudohermitian manifold of real dimension $2n+1 \geq 3$. Let $\bar{\partial}_b \colon L^2(M) \to L^2_{0,1}(M)$ be
the tangential Cauchy--Riemann operator and $\bar{\partial}_{b}^{\ast}$
the formal adjoint with respect to the volume measure 
$dv = \theta\wedge (d\theta)^n$. The Kohn-Laplacian acting on functions
is given by \( \Box_b = \bar{\partial}_b^{\ast}\bar{\partial}_b\) and the sub-Laplacian is given by \(\Delta_b=2\Re \Box_b\). 
There has been growing interest in the relation between the spectra of the sub-Laplacian and the Kohn-Laplacian and the geometric qualities of the underlying CR manifolds. 
We mention here, for example, the Lichnerowicz-type estimate for the first positive eigenvalue of the sub-Laplacian on compact manifolds with a lower bound on Ricci and torsion was studied in, e.g., \cite{ADE,BD,Gr,LL1,Chiu}. The characterization of extremal case, the Obata-type problem, was studied in, e.g., \cite{CC,LW,IV}. In particular, X.~Wang and the first author proved an Obata-type theorem in CR geometry for compact manifolds in \cite{LW} which characterizes the CR sphere (among compact manifolds) as the only extremal case in the Lichnerowicz-type estimate for the sub-Laplacian. We refer the reader to the aforementioned papers and references therein for more 
details on these problems. 

The eigenvalue problem for $\Box_b$ is more involved. It is well-known that on a \emph{non-embeddable} compact strictly pseudoconvex manifold of three-dimension, \(\mathrm{Spec}\, (\Box_b)\) contains a sequence of ``small'' eigenvalues converging rapidly to zero. In this case, we can \emph{not} define the first positive eigenvalue of $\Box_b$. In fact, by the theorems of Boutet de Monvel, Burns, and Kohn, zero is an isolated eigenvalue of $\Box_b$ if and only if $M$ is embeddable in some complex space $\mathbb{C}^N$ \cite{Burns, BdM, K}; see also \cite{BE}. Thus, for embedded manifolds, it makes sense to define and study the first positive eigenvalue $\lambda_1$ of $\Box_b$.

In \cite{CCY}, Chanillo, Chiu, and Yang proved a Lichnerowicz-type lower bound for $\lambda_1$ for three-dimensional manifolds (which are not assumed to be embedded \emph{a priori}). Their method also gives the same estimate for five dimensional case. In a preprint \cite{CW}, Chang and Wu gave a lower bound in general dimension and proved some partial results on characterizing the equality case. In \cite{LSW}, X.~Wang, the first, and the third author completely analyzed the equality case by establishing an Obata-type theorem for the Kohn-Laplacian; we refer to \cite{LSW} for more details. 

In this paper, we shall give sharp \emph{upper} bounds for $\lambda_1$ on compact strictly pseudoconvex CR manifolds embedded in $\mathbb{C}^{n+1}$. Suppose $\rho$
is a smooth strictly plurisubharmonic function on $\mathbb{C}^{n+1}$ and $\nu$ is a regular value of $\rho$ such that $M:=\rho^{-1}(\nu)$ is compact. On $M$, consider the ``usual'' pseudohermitian structure $\theta$ ``induced'' by~$\rho$:
\begin{equation}\label{b}
 \theta 
 = 
 \iota^{\ast} (i/2)(\bar{\partial}\rho - \partial \rho), 
\end{equation}
where $\iota \colon M \to \mathbb{C}^{n+1}$ is the usual embedding. This pseudohermitian structure gives rise to a volume form 
$dv: = \theta\wedge (d\theta)^n$ on $M$. Furthermore, $\rho$ induces a K\"ahler metric $\rho_{j\kbar}dz^jd\zbar^{k}$ in a neighborhood $U$ of~$M$.
Let $[\rho^{j\bar{k}}]^t$ be the inverse of $[\rho_{j\kbar}]$. For a smooth function $u$ on $U$, the length of $\partial u$ in the K\"ahler metric is given by
\begin{equation}
|\partial u|^2_{\rho} = \rho^{j\bar{k}} u_j \bar{u}_{\bar{k}}.
\end{equation}
Here we use the usual summation convention: repeated Latin indices are summing from $1$ to $n+1$. We also use $\rho^{j\bar{k}}$ and $\rho_{j\bar{k}}$ to raise and lower the indices, e.g., $u^{\bar{k}} = \rho^{ l \bar{k}} u_{l}$, so that $|\partial u|_{\rho}^2 = \bar{u}_{\bar{k}} u^{\kbar}$.
We define the following degenerate differential operator
\begin{equation}
\tilde{\Delta}_{\rho} = \left(|\partial\rho|_\rho^{-2}\rho^{j}\rho^{\kbar} - \rho^{j\kbar}\right)\partial_{j}\partial_{\kbar}.
\end{equation}
Our first result in this paper is the following sharp upper bound for $\lambda_1$.
\begin{theorem}\label{thm:special}
Let $\rho$ be a smooth strictly plurisubharmonic function defined on an open set $U$ of $\mathbb{C}^{n+1}$, $M$ a compact connected regular level set of $\rho$, and $\lambda_1$ the first positive eigenvalue of $\Box_b$ on $M$. Assume that for some $j$,
\begin{align}\label{special}
\Re \rho_{\jbar}\tilde{\Delta}_{\rho}\, \rho_j + \tfrac{1}{n}\, |\d \rho|_\rho^2 \, |\tilde{\Delta}_{\rho} \rho_j|^2 \le 0\ \text{on}\ M.
\end{align}
Then
\begin{align}\label{specialbound}
\lambda_1(M,\theta) \leq n\max_M |\d \rho|_\rho^{-2}
\end{align}
and the equality holds only if $|\partial\rho|_\rho^{2}$ is constant along $M$.
\end{theorem}
The upper bound in~\eqref{specialbound} is sharp and the equality occurs on the sphere with the standard pseudohermitian structure. Moreover, in Example~\ref{cex} below, we shall see that the condition \eqref{special} can \emph{not} be relaxed. 

Notice that condition~\eqref{special} is satisfied if there exists $j$ such that $\rho_{j\kbar l} = 0$ for all $k$ and $l$ and hence we can easily construct examples
for which Theorem~\ref{thm:special} does apply. In particular, if $\rho_{j\kbar}=\delta_{jk}$, then \eqref{special} holds. We shall show that in this case, we can improve the estimate by taking the average value of $|\partial\rho|^{-2}_{\rho}$ instead of its maximum. Thus, we define $v(M) = \int_{M}\theta \wedge (d\theta)^n$ be the volume of~$M$.
\begin{theorem}\label{thm:flat} 
Let $\rho$ be a smooth strictly plurisubharmonic function defined on an open set $U$ of $\mathbb{C}^{n+1}$, $M$ a compact connected regular level set of $\rho$, and $\lambda_1$ the first positive eigenvalue of $\Box_b$ on $M$. Suppose that $\rho_{j\kbar} = \delta_{jk}$, then
\begin{equation}\label{e:average}
\lambda_1 \leq \frac{n}{v(M)} \int_{M} |\partial\rho|^{-2}_{\rho} \theta\wedge (d\theta)^n.
\end{equation}
The equality occurs only if $|\partial\rho|_{\rho}^{2}$ is constant on $M$. If furthermore, $\rho$ is defined in the domain bounded by $M$, then $M$ must be a sphere.
\end{theorem}
The estimate \eqref{e:average} is a special case of a more general estimate in Theorem~\ref{thm:upperbound} below which provides a sharp upper bound for $\lambda_1$ in terms of the eigenvalues of the 
complex Hessian matrix~$[\rho_{j\kbar}]$.

Our main motivation for proving the upper bound in Theorem~\ref{thm:flat} comes from its application to the
eigenvalue problems on the real \emph{ellipsoids}, the compact regular level sets of a real plurisubharmonic quadratic polynomial. The ellipsoids was studied by Webster \cite{We} who showed that an ellipsoid is not biholomorphic equivalent to the sphere unless it is complex linearly equivalent to the sphere. (It is now well-known that two generic ellipsoids are not biholomorphic equivalent). The eigenvalue problem on ellipsoids was also studied by Tran and the first author \cite{LiTran}. This paper provides an upper bound for the first positive eigenvalue of $\Delta_b$ on the real ellipsoids in $\mathbb{C}^2$. We shall show that on real ellipsoids, the upper bound in Theorem~\ref{thm:flat} can be computed explicitly.
\begin{corollary}\label{cor:ellipsoid}
Let $\rho(Z)$ be a real-valued strictly plurisubharmonic homogeneous quadratic polynomial satisfying $\rho_{j\kbar} = \delta_{jk}$.
Suppose that $M = \rho^{-1}(\nu)$ ($\nu>0$) is a compact connected regular level set of $\rho$. 
Then 
\begin{align}\label{e:ellipsoid}
\lambda_1(M,\theta) \leq \lambda_1(\sqrt{\nu}\,\mathbb{S}^{2n+1},\theta_0) = n/\nu.
\end{align}
The equality occurs if and only if \((M,\theta) = (\sqrt{\nu}\,\mathbb{S}^{2n+1},\theta_0)\).
\end{corollary}

Here, \(\sqrt{\nu}\,\mathbb{S}^{2n+1}\) is the sphere $\|Z\|^2 =\nu$ and \(\theta_0 = \iota^{\ast} (i\bar{\partial} \|Z\|^2)\) is the ``standard'' pseudohermitian structure on the sphere.

The paper is organized as follows. In Section 2, we shall give two simple formulas for the Kohn-Laplacian on compact real hypersurfaces in complex 
manifolds. These formulas allow us to compute $\Box_b$ explicitly in terms of the defining function~$\rho$; see Proposition~\ref{prop:kl}. These formulas will
be crucial for the latter sections. In Section 3, we shall prove a general estimate for $\lambda_1(\Box_b)$ and Theorem~\ref{thm:special}. In Section 4, we shall give a sharp upper bound for $\lambda_1$ in terms of the eigenvalues of the complex Hessian~$[\rho_{j\kbar}]$, implying the estimate in Theorem~\ref{thm:flat}, and prove the characterization of equality case. We also give a family of examples (beside the ellipsoids) where we can apply this bound. These examples also show that the condition \eqref{special} in Theorem~\ref{thm:special} can not be relaxed. In Section 5, we shall compute the bound in Theorem~\ref{thm:flat} explicitly in the case of ellipsoids, proving Corollary~\ref{cor:ellipsoid}.

\section{The Kohn-Laplacian on compact real hypersurfaces}
\def\npbar{\overline{n+1}}
In this section, we shall give two formulas for $\Box_b$ on a compact regular level set of a 
K\"ahler potential $\rho$ in terms of $\partial \rho$ and the metric $\rho_{j\bar{k}}dz^jdz^{\bar{k}}$. First, let us start with a compact real hypersurface in $\mathbb{C}^{n+1}$ arising as a regular level set of a strictly plurisubharmonic function $\rho$:
\begin{equation}
M = \rho^{-1}(\nu):=\{Z \in U \colon \rho(Z) = \nu\}.
\end{equation}
Here $\rho$ is smooth on a neighborhood $U$ of $M$ and $d\rho\ne 0$ along~$M$. We assume that the complex Hessian $H(\rho): = [\rho_{j\kbar}]$ is positive definite and thus $\rho$ defines a K\"ahler metric $\rho_{j\bar{k}}dz^jd\bar{z}^{k}$ on $U$. Let $[\rho^{j\bar{k}}]^t$ be the inverse of $H(\rho)$. For a smooth function $u$ on $U$, the length of $\partial u$ in the K\"ahler metric is then given by
\begin{equation}
|\partial u|^2_{\rho} = \rho^{j\bar{k}} u_j \bar{u}_{\bar{k}}.
\end{equation}
We shall always equip $M$ with the pseudohermitian structure $\theta$ ``induced'' by $\rho$:
\begin{equation}
\theta = \iota^{\ast} (i/2)(\bar{\partial} \rho - \partial \rho).
\end{equation}
For local computations, it is convenient to work in the local admissible holomorphic coframe $\{\theta^{\alpha} \colon \alpha = 1,2,\dots, n\}$ on $M$ given by
\begin{equation}\label{e:2.4}
\theta^{\alpha} = dz^{\alpha} - i h^{\alpha} \theta, 
\quad
h^{\alpha} = |\partial \rho|_{\rho}^{-2}\rho^{\alpha} 
=
|\partial \rho|_{\rho}^{-2}\rho_{\jbar}\rho^{\alpha\jbar},
\quad 
\alpha = 1,2\dots n.
\end{equation}
This admissible coframe is valid when $\rho_{n+1} \ne 0$. It is shown by Luk and the first author \cite[p. 679]{LL} that at the point $p$ with $\rho_{n+1} \ne 0$,
\begin{equation}
d\theta = ih_{\alpha\bar{\beta}} \theta^{\alpha}\wedge \theta^{\bar{\beta}},
\end{equation}
where the Levi matrix $[h_{\alpha\bar{\beta}}]$ is given explicitly:
\begin{equation}
h_{\alpha\betabar}=\rho_{\alpha \betabar}-\rho_\alpha \d_{\betabar}\log \rho_{n+1}-\rho_{\betabar}\d_{\alpha}\log \rho_{\overline{n+1}}+\rho_{n+1\overline{n+1}}
\frac{\rho_\alpha \rho_{\betabar}}{|\rho_{n+1}|^2}.
\end{equation}
We can check directly that the inverse $[h^{\gamma\bar{\beta}}]$ of the Levi matrix is given by
\begin{equation}
h^{\gamma \betabar}
=
\rho^{\gamma\betabar}- \frac{\rho^\gamma \rho^{\betabar}}{|\d \rho|^2_\rho},
\quad
\rho^{\gamma} = \sum_{k=1}^{n+1}\rho_{\kbar} \rho^{\gamma\kbar}.
\end{equation}
We use the Levi matrix and its inverse to lower and raise the Greek indices; repeated Greek indices are summing from $1$ to $n$. The Tanaka-Webster 
covariant derivatives are given by
\begin{equation}
\nabla_{\alpha}\nabla_{\bar{\beta}}f 
=
Z_{\alpha} Z_{\betabar} f - \omega_{\bar{\beta}}{}^{\bar{\sigma}}(Z_{\alpha}) Z_{\bar{\sigma}} f
\end{equation}
where $\{Z_{\alpha}\}$ is the holomorphic frame dual to $\{\theta^{\alpha}\}$ and $\omega_{\bar{\beta}}{}^{\bar{\sigma}}$ are the connection forms. More precisely,
\begin{equation}
Z_{\alpha}
=
\frac{\d}{\d z^{\alpha}} - \frac{\rho_{\alpha}}{\rho_{n+1}} \frac{\d}{\d z_{n+1}},
\end{equation}
and the Tanaka-Webster connection forms are computed in \cite{LL}; see also \cite{We}.
\begin{equation}
\omega_{\betabar\alpha}
=
(Z_{\bar{\gamma}} h_{\alpha\betabar} - h_{\betabar}h_{\alpha\gamabar}) \theta^{\gamabar} + h_{\alpha}h_{\gamma\betabar}\theta^{\gamma} + ih_{\alpha\bar{\sigma}}Z_{\betabar} h^{\bar{\sigma}} \theta,\quad h_\alpha =h_{\alpha \betabar} h^{\betabar}.
\end{equation}
Also, the Reeb vector field is given by
\begin{equation}
 T=i \sum_{j=1}^{n+1}\left(h^j \frac{\d }{ \d z^j}-h^{\jbar} \frac{\d }{ \d \zbar^j}\right),\quad h^j= \frac{\rho^j }{|\d \rho|_\rho^2}.
 \end{equation}
The formula \eqref{e:kohnformula} below, expressing $\Box_b$ in terms of $\rho$, will be crucial for our analysis.
\begin{proposition}\label{prop:kl}
Let $U$ be an open set in a K\"ahler manifold $X$ and $\rho$ a K\"ahler potential on $U$. Let $M$ be a smooth, compact, connected, regular level set of $\rho$,
\(\theta=\frac{i}{2}(\bar{\partial} \rho - \partial \rho)\), and $\Box_b$ the Kohn-Laplacian defined on $M$ with respect to $dv=\theta\wedge (d\theta)^n$.
\begin{enumerate}[(i)]
 \item If $f$ is a smooth function on $U$, then the following identity holds on $M$.
\begin{equation}\label{e:kohnformula}
\Box_b f
=
- \trace (i\partial \bar{\partial} f)
+ |\partial \rho|_\rho^{-2} \langle \partial \bar{\partial} f, \partial \rho \wedge \bar{\partial} \rho \rangle
+ n|\partial \rho|_\rho^{-2}\langle \partial \rho, \bar{\partial} f \rangle,
\end{equation}
 \item Suppose that $(z^1,z^2,\dots, z^{n+1})$ is a local coordinate system on an open set $V$. Define the vector fields
 \begin{equation}
X_{jk} = \rho_{k}\partial_j - \rho_{j}\partial_k, 
\quad
X_{\bar{j}\bar{k}} = \overline{X_{jk}}.
\end{equation}
Then the following holds on $M \cap V$.
\begin{equation}\label{e:gengel}
 \Box_b f
 =
 -\frac{1}{2}|\partial \rho|_\rho^{-2} \rho^{p\bar{k}} \rho^{q\bar{j}} X_{pq}X_{\bar{j}\bar{k}} f.
\end{equation}
\end{enumerate}
\end{proposition}
\begin{remark}\rm
\begin{enumerate}[(a)]
\item The trace operator is taken with respect to the K\"ahler form and thus $- \trace (i\partial \bar{\partial} f)$ is the Laplace-Beltrami operator acting on $f$. In local coordinates, \eqref{e:kohnformula} can be written as
\begin{equation}
\Box_b f
 = 
 \left(|\partial \rho|_\rho^{-2} \rho^{k}\rho^{\bar{j}} -\rho^{\bar{j} k}\right) f_{\bar{j} k } + n |\partial \rho|_\rho^{-2} \rho^{\bar{k}} f_{\bar{k}}.
\end{equation}
\item Formulas~\eqref{e:gengel} and \eqref{e:kohnformula} are generalizations of two formulas for the Kohn-Laplacian on the sphere appeared in \cite{Geller1980}. This paper also studies the Kohn-Laplacian for forms on the sphere (with volume element induced from $\mathbb{C}^{n+1}$). Notice that the fields $X_{jk}$ are \emph{tangential} Cauchy-Riemann vector fields on $M$ generating $T^{1,0}$ at each point.
\end{enumerate}
\end{remark}
\begin{proof} We first prove (i). It is well-known \cite{L} that the Kohn Laplacian acting on function can be given locally by
\begin{align}
-\Box_b f 
 =
 h^{\betabar\alpha} \nabla_{\alpha}\nabla_{\bar{\beta}}f.
\end{align}
Thus, we can work in a local coordinate $(z^1,z^2,\dots, z^{n}, w=z^{n+1})$ on $X$ and assume that
$\rho_w = \partial_w \rho \ne 0$. Choose the local frame and coframe as in \eqref{e:2.4}. Notice that
\begin{align}
Z^{\betabar}
 = h^{\alpha\betabar} Z_\alpha 
 = 
h^{\alpha\betabar} \d_{\alpha}- h^{\alpha\betabar} \frac{\rho_\alpha }{ \rho_{n+1}} \d_{n+1}
 = \rho^{k\betabar}\d_k-\frac{\rho^{\betabar}}{ |\d \rho|_\rho^2}\rho^k \d_k.
 \end{align}
Therefore,
\begin{align}
-\Box_b f 
 = &
Z^{\betabar}Z_{\betabar} f -n h^{\bar{\sigma}} f_{\bar{\sigma}}\notag \\
 =&
 \left[ \rho^{k\betabar}\d_k - |\d \rho|_\rho^{-2} \rho^{\betabar} \rho^k \d_k \right]
 \left[f_{\betabar} - \frac{\rho_{\betabar}}{\rho_{\wbar}}f_{\wbar}\right] -n h^{\bar{\sigma}} f_{\bar{\sigma}} \notag\\
 = & 
 \rho^{k\betabar}f_{\betabar k} - \frac{\rho_{\betabar}\rho^{k\betabar} f_{\wbar k}}{\rho_{\wbar}} -\rho^{k\betabar} f_{\wbar} \left[\frac{\rho_{\wbar} \rho_{\betabar k} - \rho_{\betabar} \rho_{\wbar k}}{\rho_{\wbar}^2}\right] \notag\\
 & 
 - \frac{\rho^{k}\rho^{\betabar}f_{\betabar k}}{|\partial \rho|_\rho^2} 
 + \frac{\rho^{k}\rho^{\betabar}\rho_{\betabar}f_{\wbar k}}{|\partial \rho|_\rho^2 \rho_{\wbar}}
 +\frac{\rho^{k}\rho^{\betabar}f_{\wbar}}{|\partial \rho|_\rho^2}\left[\frac{\rho_{\wbar}\rho_{\betabar k} - \rho_{\betabar}\rho_{\wbar k}}{\rho_{\wbar}^2}\right] \notag\\
 & 
 -\frac{n\rho^{\bar{k}}f_{\bar{k}}}{|\partial \rho|_\rho^2} + \frac{nf_{\wbar}}{\rho_{\wbar}}.
\end{align}
Here we use summation convention: $k$ runs from $1$ to $n+1$ and $\beta$ runs $1$ to $n$. Simplifying the right hand side, we easily obtain
\begin{equation}
-\Box_b f
=
 \left( \rho^{\bar{j} k} - |\partial \rho|_\rho^{-2} \rho^{k}\rho^{\bar{j}}\right) f_{\bar{j} k } - n |\partial \rho|_\rho^{-2} \rho^{\bar{k}} f_{\bar{k}},
\end{equation}
which is clearly equivalent to \eqref{e:kohnformula}.

To prove (ii), we notice that
\begin{equation}
X_{\bar{j}\bar{k}} f
= \rho_{\bar{k}} f_{\bar{j}} - \rho_{\bar{j}} f_{\bar{k}}.
\end{equation}
Therefore,
\begin{align}
X_{pq}X_{\bar{j}\bar{k}} f
 = &
 \rho_q \rho_{\bar{k}} f_{\bar{j}p} + \rho_{q}\rho_{\bar{k}p}f_{\bar{j}} - \rho_{q}\rho_{\bar{j}p}f_{\bar{k}} - \rho_{q}\rho_{\bar{j}} f_{\bar{k} p} \notag \\
 & 
 -\rho_{p}\rho_{\bar{k}} f_{\bar{j} q} - \rho_{p}\rho_{\bar{k}q} f_{\bar{j}}
 +\rho_{p}\rho_{\bar{j}} f_{\bar{k}q} + \rho_{p}\rho_{\bar{j}q}f_{\bar{k}}.
\end{align}
Contracting both sides with $\rho^{p\bar{k}}\rho^{q\bar{j}}$, using (i), we easily obtain (ii).
\end{proof}

\section{An estimate for eigenvalues and proof of Theorem~\ref{thm:special}}

We denote by $S\colon L^2(M) \to \ker \Box_b$ ($=\ker \bar{\partial}_b$) the Szeg\H{o} orthogonal projection with respect to the volume measure
$dv:=\theta\wedge (d\theta)^n$. It is well-known that if $M$ is embeddable, then $\mathrm{Spec}(\Box_b)$ consists of zero and a sequence of 
point eigenvalues $\{\lambda_k\}$ increasing to infinity. The positive eigenvalues of $\Box_b$ are of finite multiplicity and eigenfunctions 
are smooth \cite{BG, BE}. Furthermore, we have the following orthogonal decomposition:
\begin{align}
L^2(M,dv)
=
\bigoplus_{k=0}^{\infty} E_k, \quad E_0 = \ker \Box_b.
\end{align}
Note that $E_0$ is of infinite dimension.
\begin{theorem}\label{thm:generalestimate}
Let $(M,\theta)$ be an embedded compact strictly pseudoconvex pseudohermitian manifold and $0=\lambda_0<\lambda_1<\lambda_2<\cdots <\lambda_k <\cdots $ the eigenvalues for $\Box_b$.
Define
\begin{equation} 
 m(a)=\inf\left\{\left|a - \frac{1}{\lambda_k}\right|^2: k\in \NN\right\}, 
 \quad
 M(a) = \sup\left\{\left|a - \frac{1}{\lambda_k}\right|^2: k\in \NN\right\}.
\end{equation} 
Then for any $a\in \mathbb{R}$, any function $u \not \in \ker\Box_b$,
\begin{equation}\label{e:mainestimate}
(m(a) - a^2) \|\Box_b u\|^2 \leq \|u - S(u)\|^2 - \int_M |\bar{\partial}_b u|^2 \leq (M(a) - a^2)\|\Box_b u\|^2.
\end{equation}
\end{theorem}
\begin{proof} Let $E_{k}$ be the eigenspace of $\Box_b$
associated to the eigenvalue $\lambda_k$. Then $m_k:=\hbox{dim}(E_k)<\infty$. 
Let $\{f_{k, j}\}_{j=1}^{m_k}$ be an orthonormal basis for $E_{k}$. For any $k,\ell$, using integration by parts, we obtain
\begin{equation}
\int_M (\Box_b u - \lambda_k u) \bar{f}_{k,\ell}
= \int_M ( u\overline{\Box_b f_{k,\ell}} - \lambda_k u \bar{f}_{k,\ell})
= \int_M u(\overline{\Box_b f_{k,\ell} - \lambda_k f_{k,\ell}}) = 0.
\end{equation}
This implies that for any real number $a$,
\begin{equation}
\langle u-a \Box_b u, f_{k,\ell}\rangle = -(a-1/\lambda_k)\langle \Box_b u,f_{k,\ell}\rangle.
\end{equation}
Therefore, since $\Box_b u \in (\ker \Box_b)^{\perp}$,
\begin{align}
M(a)\|\Box_b u\|^2
 = &
 \sum_{k=1}^{\infty}\sum_{\ell=1}^{m_k}M(a)\left|\langle \Box_b u, f_{k,\ell}\rangle\right|^2 \label{e:1}\\
 \geq &
 \sum_{k=1}^{\infty}\sum_{\ell=1}^{m_k} \left|a-\frac{1}{\lambda_k}\right|^2\left|\langle \Box_b u, f_{k,\ell}\rangle\right|^2 \\
 = &
 \sum_{k=1}^{\infty}\sum_{\ell=1}^{m_k} \left|\langle u-a\,\Box_b u, f_{k,\ell}\rangle\right|^2 \\
 = &
 \|u-a\,\Box_b u\|^2 - \|S(u-a\,\Box_b u)\|^2 \\
 = & 
 \|u\|^2 + a^2 \|\Box_b u\|^2 - 2a \int_M \bar{u}\, \Box_b u - \|S(u)\|^2.
\end{align}
Here we have used $\|S(u-a\, \Box_b u)\|^2 = \|S(u)\|^2$. We conclude that
\begin{equation}
(M(a)-a^2) \|\Box_b u\|^2 \geq \|u-S(u)\|^2 - 2a \int_M |\bar{\partial}_b u|^2.
\end{equation}
This proves the second inequality. The first inequality can be proved similarly.
\end{proof}
The following two corollaries are undoubtedly known, but we can not find in the literature.
\begin{corollary}\label{cor:32} Let $(M,\theta)$ be as in Theorem~\ref{thm:generalestimate}, then
\begin{equation}
\lambda_1 
= 
\inf\left\{\bigl\| \Box_b u\bigr\|^2 \colon \int_M |\bar{\partial}_b u|^2 = 1\right\}
=
 \inf\left\{\int_M |\bar{\partial}_b u|^2 \colon \|u-S(u)\|^2 =1\right\}.
\end{equation}
\end{corollary}
\begin{proof} For any $a> \frac{1}{\lambda_1}$, we have
\begin{equation}
m(a) = \left|a - \frac{1}{\lambda_1}\right|^2.
\end{equation}
From Theorem~\ref{thm:generalestimate}, we have for any $u$ with $\int_M |\bar{\partial}_b u|^2 = 1$,
\begin{equation}
\left[\left|a - \frac{1}{\lambda_1}\right|^2 - a^2\right]\bigl\| \Box_b u\bigr\|^2
\leq \|u-S(u)\|^2 - 2a.
\end{equation}
This is equivalent to
\begin{equation}
\left(\frac{1}{\lambda_1}\right)^2
-2a\left(\frac{1}{\lambda_1} - \frac{1}{\| \Box_b u\|^2}\right) 
\leq
\frac{ \|u-S(u)\|^2}{\| \Box_b u\|^2}.
\end{equation}
Letting $a\to +\infty$, we easily obtain 
\begin{equation}
\lambda_1 \leq \| \Box_b u\|^2.
\end{equation}
Since $u$ is arbitrary, we conclude that 
\begin{equation}
\lambda_1 \leq \inf\left\{\bigl\| \Box_b u\bigr\|^2 \colon \int_M |\bar{\partial}_b u|^2 = 1\right\}.
\end{equation}
The reverse inequality is trivial. 

To prove the second we take $a = \frac{1}{2}\frac{1}{\lambda_1}$ and notice that $M(a) = a^2$. Then from
Theorem~\ref{thm:generalestimate}, we deduce that for any $u$ satisfying $\|u-S(u)\|^2 = 1$,
\begin{equation}
0=(M(a)-a^2) \|\Box_b u\|^2 \geq \|u-S(u)\|^2 - 2a \int_M |\bar{\partial}_b u|^2
=1- 2a \int_M |\bar{\partial}_b u|^2.
\end{equation}
Hence,
\begin{equation}
\lambda_1
=
\frac{1}{2a} 
\leq \int_M |\bar{\partial}_b u|^2.
\end{equation}
The proof of the reverse inequality is simple and omitted.
\end{proof}
\begin{corollary}
 Let $(M,\theta)$ be as in Theorem~\ref{thm:generalestimate}, then for any function $u$,
\begin{equation}
\|u-S(u)\|\cdot \|\Box_b u\|
\geq
\int_M |\bar{\partial}_b u|^2.
\end{equation}
\end{corollary}
\begin{proof}
Without lost of generality, we may assume that $\int_M |\bar{\partial}_b u|^2 = 1$.
For each $k$, we take $a_k = \frac{1}{2}\left(\frac{1}{\lambda_k}+\frac{1}{\lambda_{k+1}}\right)$. Clearly,
\begin{equation}
m(a_k) = \left|a_k - \frac{1}{\lambda_k}\right|^2 = \left|a_k - \frac{1}{\lambda_{k+1}}\right|^2
\end{equation}
By Theorem~\ref{thm:generalestimate}, we have
\begin{equation}
\left[\left|a_k - \frac{1}{\lambda_k}\right|^2 - a_k^2\right] \|\Box_b u\|^2
\leq \|u-S(u)\|^2 - 2a_k.
\end{equation}
By direct calculation, we have that
\begin{equation}\label{e:syli}
\left(\frac{1}{\lambda_{k}} - \frac{1}{\|\Box_b u\|^2}\right)\left(\frac{1}{\lambda_{k+1}} - \frac{1}{\|\Box_b u\|^2}\right)
\geq
\frac{1}{\|\Box_b u\|^4} - \frac{\|u-S(u)\|^2}{\|\Box_b u\|^2}.
\end{equation}
By Corollary~\ref{cor:32}, $\lambda_1 \leq \|\Box_b u\|^2$. Moreover, $\lambda_k \to \infty$ as $k\to \infty$. We deduce that there exists $k_0$ such that 
\begin{equation}
\lambda_{k_0} \leq \|\Box_b u\|^2 < \lambda_{k_0+1}.
\end{equation}
 Therefore, \eqref{e:syli} with $k=k_0$ implies that
\begin{equation}
\frac{1}{\|\Box_b u\|^4} - \frac{\|u-S(u)\|^2}{\|\Box_b u\|^2} \leq 0.
\end{equation}
This completes the proof.
\end{proof}
\begin{proposition}\label{prop:Bz}
Let $(M,\theta)$ be a compact strictly pseudoconvex pseudohermitian manifold. If there is a smooth non-CR function $f$ on $M$ such that $|\Box_b f|^2\le B(z) \Re \fbar \Box_b f$ for some non-negative function $B$ on $M$, then
\begin{equation}
\lambda_1\le \max_M B(z).
\end{equation}
If the equality holds, then $B$ must be a constant.
\end{proposition}
\begin{proof}
 Since $|\Box_b f|^2\le B(z) \Re \fbar \Box_b f$, by Corollary \ref{cor:32}, 
\begin{equation}\label{estimate:3.26}
 \lambda_1 \int_M \fbar \Box_b f
 \leq
 \int_M |\Box_b f|^2
 \leq
 \int_M B(z) \Re (\fbar \Box_b f).
\end{equation}
By the Mean Value Theorem of the integral, there is $z_0\in M$ such that
\begin{equation}
0
\leq
\int_M (B-\lambda_1) \Re (\fbar \Box_b f)
=
(B(z_0)-\lambda_1)\int_M \fbar \Box_b f 
=
(B(z_0)-\lambda_1)\int_M |\bar{\partial}_b f|^2.
\end{equation}
This implies
\begin{equation}
\lambda_1\le B(z_0)\le \max_M B(z).
\end{equation}
It is clear that if $\lambda_1=\max_M B$ then $B$ is a constant. 
\end{proof}
We end this section by proving the Theorem~\ref{thm:special}.
\begin{proof}[Proof of Theorem~\ref{thm:special}]
By the condition \eqref{special} and the expression for the Kohn-Laplacian given by \eqref{e:kohnformula}, we have
\begin{equation}
\Box_b \rho_j
 =
 \tilde{\Delta}_{\rho} \rho_j + n|\partial\rho|_\rho^{-2} \rho^{\kbar} \rho_{j\kbar}
 =\tilde{\Delta}_{\rho} \rho_j + n|\partial\rho|_\rho^{-2} \rho_{j}.
 \end{equation}
 Then
 \begin{equation}
 |\Box_b \rho_j|^2=\frac{n}{|\d \rho|_\rho^2}\Re\left(\rho_{\jbar} \Box_b \rho_j +\rho_{\jbar} \tilde{\Delta}_{\rho}\rho_j+\tfrac{1}{n}|\d \rho|_\rho^2\, |\tilde{\Delta}_{\rho} \rho_j|^2\right)
 \leq
 \frac{n}{|\d \rho|_\rho^2}\Re \left( \rho_{\jbar} \Box_b \rho_j \right).
 \end{equation}
Applying Proposition~\ref{prop:Bz} with $B(z)=n |\d \rho|^{-2}_\rho$, we obtain
 \begin{equation}
 \lambda_1
 \leq
 n \max_M |\d \rho|_\rho^{-2}.
 \end{equation}
The equality holds only if $|\d \rho|_\rho$ is a constant on $M$. The proof of Theorem~\ref{thm:special} is complete.
\end{proof}
\section{Proof of Theorem~\ref{thm:flat}}

The following theorem gives a sharp upper bound for $\lambda_1(\Box_b)$ in terms of the eigenvalues of the complex Hessian matrix~$[\rho_{j\kbar}]$ and the length of $\partial\rho$. This theorem implies the estimate in Theorem~\ref{thm:flat}.
\begin{theorem}\label{thm:upperbound}
Let $\rho$ be a smooth strictly plurisubharmonic function defined on an open set $U$ of $\mathbb{C}^{n+1}$, $M$ a compact connected regular level set of $\rho$, and $\lambda_1$ the first positive eigenvalue of $\Box_b$ on $M$.
Let $r (z)$ be the spectral radius of the matrix $[\rho^{j\kbar}(z)]$ and \(s (z) = \mathrm{trace}\, [\rho^{j\kbar}] - r (z)\).
Then
\begin{equation}\label{e:upperbound}
\lambda_1 \leq \frac{n^2\int_{M} r (z)|\partial\rho|^{-2}_{\rho}}{\int_M s (z)}.
\end{equation}
\end{theorem}
Here the spectral radius of a square matrix is the maximum of the moduli of its eigenvalues.
\begin{proof} First, we define
\begin{equation}
C_j
=
 \int_{M} \frac{|\rho^{\jbar}|^2}{|\partial\rho|_{\rho}^4},
\quad 
D_{j}
=
\int_{M} \left(\rho^{j\jbar} - \frac{|\rho^{\jbar}|^2}{|\partial\rho|_{\rho}^2}\right).
\end{equation}
 From Proposition~\ref{prop:kl}, we can compute 
\begin{equation}
\Box_b \zbar^j 
= 
n|\partial\rho|^{-2}_{\rho}\rho^{\jbar}.
\end{equation}
Therefore,
\begin{equation}
\|\Box_b \zbar^j \|^2
=
n^2\int_{M} \frac{|\rho^{\jbar}|^2}{|\partial\rho|_{\rho}^4}
=
n^2C_j.
\end{equation}
We can also compute
\begin{equation}
|\bar{\partial}_b \bar{z}^j|^2
=
\delta_{j\alpha}\delta_{j\beta}\left(\rho^{\alpha\bar{\beta}} - \frac{\rho^{\alpha}\rho^{\bar{\beta}}}{|\partial\rho|_{\rho}^2}\right)
=
\rho^{j\jbar} - \frac{|\rho^{\jbar}|^2}{|\partial\rho|_{\rho}^2}.
\end{equation}
Here without lost of generality, we assume $j\ne n+1$. Therefore,
\begin{equation}
\int_M |\bar{\partial}_b \bar{z}^j|^2
=
D_j.
\end{equation}
Thus, from Corollary~\ref{cor:32} above, we obtain for all $j$,
\begin{equation}\label{e:lambdaest}
\lambda_1
\leq 
n^2C_j/D_j.
\end{equation}

Next, observe that $1/r (z)$ is the smallest eigenvalue of the Hermitian matrix $[\rho_{j\kbar}(z)]$, and thus, for all $(n+1)$-vector $v^j$,
\begin{equation}
\frac{1}{r (z)} \sum_{j=1}^{n+1} |v^j|^2 \leq v^j\rho_{j\kbar} v^{\kbar}.
\end{equation}
Plugging $v^j = \rho^j$ into the inequality, we easily obtain \(\sum\limits_{j=1}^{n+1}|\rho^j|^2\leq r(z){|\partial\rho|^2_{\rho}} \).
Consequently
\begin{equation}\label{e:cest}
\sum_{j} C_j
=
\sum_{j=1}^{n+1} \int_M \frac{|\rho^j|^2}{|\partial\rho|^4_{\rho}}
\leq 
\int_M r (z)|\partial\rho|^{-2}_{\rho}, 
\end{equation}
and therefore,
\begin{equation}\label{e:dest}
\sum_{j} D_j 
= 
\sum_{j=1}^{n+1} \int_{M} \left(\rho^{j\jbar} - \frac{|\rho^j|^2}{|\partial\rho|^2_{\rho}}\right)
\geq 
\int_{M} \left[\trace [\rho^{j\kbar}] - r (z)\right]
=\int_M s (z).
\end{equation}
Thus, from \eqref{e:lambdaest}, \eqref{e:cest}, and \eqref{e:dest}, we obtain
\begin{equation}
\lambda_1 \leq n^2\min_j (C_j/D_j) \leq \frac{n^2\sum_j C_j}{\sum_j D_j} 
=
 \frac{n^2\int_{M} r (z)|\partial\rho|^{-2}_{\rho}}{\int_M s (z)}.
\end{equation}
The proof is complete.
\end{proof}
\begin{proof}[Proof of Theorem~\ref{thm:flat}]
Since $\rho_{j\kbar} = \delta_{jk}$, we have \(r (z) = 1\) and \(s (z) = n\).
Therefore, by Theorem~\ref{thm:upperbound},
\begin{equation}
\lambda_1 
\leq
 \frac{n^2\int_{M} r (z)|\partial\rho|^{-2}_{\rho}}{\int_M s (z)}
=\frac{n}{v(M)} \int_M |\partial\rho|^{-2}_{\rho}.
\end{equation}
which proves the inequality.

Next we suppose that $\lambda_1 = \frac{n}{v(M)} \int_M |\partial\rho|_{\rho}^{-2}$. We shall show that $|\partial\rho|^{2}_{\rho}$ is constant along $M$. Put
\begin{equation}
b_j = n^{-1}\Box_b \zbar^j = |\partial\rho|_{\rho}^{-2}\rho_j.
\end{equation}
Then by inspecting the proof of Theorem~\ref{thm:generalestimate} above, in particular, the estimate~\eqref{e:1}, we have for all $j$,
\begin{equation} \label{e:c}
\langle b_j, f_{k,\ell}\rangle = 0,\quad \text{for all}\ \ell, \ \text{for all} \ k\ne 1.
\end{equation} 
Thus, $b_j \perp \ker\Box_b$ and \eqref{e:c} imply that $b_j\in E_1$ (the eigenspace corresponding to $\lambda_1$). Therefore,
\begin{align}
\Box_b b_j = \lambda_1 b_j.
\end{align}
Recall that $\Box_b \zbar^j = nb_j$. We then deduce that
\begin{align}
\Box_b\left[\zbar^j - \frac{n}{\lambda_1}\frac{\rho_{j}}{|\partial\rho|^2_{\rho}} \right]=0.
\end{align}
Hence, $\zbar^j -n\rho^{\jbar}/(\lambda_1 |\partial \rho|^2_{\rho})$ restricted to $M$ is a CR function. Since $X_{\bar{l}\bar{k}}$ is a tangential CR vector fields on $M$, we have
\begin{align}
X_{\bar{l}\bar{k}}\left[\zbar^j - \frac{n}{\lambda_1}\frac{\rho_j}{|\partial\rho|_{\rho}^2} \right]=0.
\end{align}
By direct calculation, this is equivalent to 
\begin{equation}
\frac{n}{\lambda_1} \rho_j X_{\bar{l}\bar{k}} (|\partial\rho|_{\rho}^2)/ |\partial \rho|_{\rho}^4
=
\left(1- \frac{n}{\lambda_1|\partial \rho|_{\rho}^2}\right)\left(\rho_{\bar{l}}\delta_{jk} - \rho_{\bar{k}} \delta_{jl}\right).
\end{equation}
Since $M$ is compact, there exists point $x\in M$ such that
\begin{equation}
|\partial\rho(x)|_{\rho}^2 = \max_{M}|\partial\rho|_{\rho}^2.
\end{equation}
At the maximum point $x$, we also have $X_{\bar{l}\bar{k}} |\partial\rho|_{\rho}^4= 0$. Thus,
\begin{equation}
\left[1 - \frac{n}{\lambda_1 |\partial\rho|_{\rho}^2}\right] \left(\rho_{\bar{l}}\delta_{jk} - \rho_{\bar{k}} \delta_{jl}\right)
=0 \quad \text{at} \ x.
\end{equation}
 Since $\partial \rho(x) \ne 0$, we can assume that $\rho_{\bar{1}}(x) \ne 0$. Taking $j=k=2$, we have at $x$,
\begin{equation}
 1 - \frac{n}{\lambda_1 |\partial\rho|_{\rho}^2}
=
 0.
 \end{equation}
 Therefore,
 \begin{equation}
\min |\partial\rho|_{\rho}^{-2}
=
|\partial \rho(x)|_{\rho}^{-2}
 =
\frac{\lambda_1}{n}
=
\frac{1}{v(M)}\int_{M} |\partial\rho|_{\rho}^{-2}.
 \end{equation}
 As the right most term is the average of $|\partial \rho|_{\rho}^{-2}$ on $M$, we deduce from above that $|\partial \rho|_{\rho}^{-2}$ must be constant on~$M$. 
 
Finally, suppose that $|\partial\rho|_{\rho}^2$ is constant along $M$ and $\rho$ extends to the domain bounded by $M$ and satisfies $\rho_{j\kbar} = \delta_{jk}$ on the domain. We shall show in the lemma below that $M$ must be a sphere and complete the proof of Theorem~\ref{thm:flat}.
\end{proof}
 \begin{lemma}\label{lem:c} Let $M$ be a compact connected regular level set of $\rho$ which bounds a domain $D$. Suppose that $\rho_{j\kbar} = \delta_{jk}$ on~$D$. If $|\partial \rho|_{\rho}^2$ is constant on $M$, then $M$ must be a sphere. 
 \end{lemma}
\begin{proof}[Proof of Lemma~\ref{lem:c}] The proof is an application of Serrin's theorem \cite[Theorem~1]{Ser}.
 Let $D$ be the domain with $M$ is its boundary. Define $u = \rho - \nu$ on a neighborhood of $\overline{D}$. Since $M$ is smooth and the function $u$ satisfies $\Delta u=-4(n+1)$ in $D$, $u=0$ on $\d D$,
and the normal derivative \(\d u/\d\mathbf{n} = 2|\d \rho|_{\rho}\) is a constant on $\d D$ by assumption, we can apply the Serrin's theorem to conclude that $M$ is a standard sphere.
\end{proof}
We end this section by the following example which gives a sharp upper bound on the family of compact level sets
of K\"{a}hler potentials of Fubini-Study metric. This example also shows that the condition \eqref{special} in Theorem~\ref{thm:special} can not be relaxed.
\begin{example}\label{cex}\rm 
Let $\rho$ be a strictly plurisubharmonic function of the form
\begin{equation} \label{e:fsp}
\rho(Z) = \log (1+\|Z\|^2) + \psi(Z,\bar{Z}),
\end{equation} 
where $\psi$ is a real-valued pluriharmonic function. We suppose that $\rho$ is defined and proper in some domain $U \subset \mathbb{C}^{n+1}$ (e.g., if $\psi = -\log |z_1|$, then $\rho$ is defined and proper on $(\mathbb{C}\setminus\{0\}) \times \mathbb{C}^n$). 

Observe that
\begin{equation} 
\rho_{j\kbar} 
=
\frac{1}{1+\|Z\|^2}\left(\delta_{jk} - \frac{\zbar^{j} z^k}{1+\|Z\|^2}\right), \quad
\rho^{j\kbar}
=
(1+\|Z\|^2)\left(\delta_{jk} + \zbar^{k} z^j\right),
\end{equation} 
By a routine calculation, we see that the characteristic polynomial of $[\rho^{j\kbar}]$ is 
\begin{equation}
 P_{[\rho^{j\kbar}]}(\lambda) = (1+\|Z\|^2-\lambda)^n\bigl[(1+\|Z\|^2)^2-\lambda\bigr].
\end{equation}
Thus, the spectral radius of $[\rho^{j\kbar}]$ is $r (Z) = (1+\|Z\|^2)^2$ and $s (Z) = \trace [\rho^{j\kbar}] - r (Z) = n(1+\|Z\|^2)$. By Theorem~\ref{thm:upperbound}, if $M$ is a compact, connected, regular level set of $\rho$, then
\begin{equation} 
\lambda_1
\leq 
\frac{n\int_M (1+\|Z\|^2)^2 |\partial\rho|_{\rho}^{-2}}{\int_M (1+\|Z\|^2)} \leq n\max_{M} (1+\|Z\|^2)|\partial\rho|_{\rho}^{-2}.
\end{equation} 
Notice that if $\psi=0$ and then $M_{\nu}:=\rho^{-1}(\nu)$ is the sphere $\|Z\|^2 = e^{\nu}-1$ with
\begin{equation}
\theta
=
ie^{-\nu}\sum_{j=1}^{n+1} \left(z^j d\zbar^j - \zbar^j dz^j\right).
\end{equation} 
Moreover, $|\partial\rho|^2_{\rho} = e^{\nu}-1$ on $M_{\nu}$ and $\lambda_1 = n e^{\nu}/(e^{\nu}-1)$. Therefore, the condition \eqref{special} in Theorem~\ref{thm:special} can not be relaxed.
\end{example}
\section{The real ellipsoids: Proof of Corollary~\ref{cor:ellipsoid}}
The proof of Theorem~\ref{cor:ellipsoid} follows from Theorem~\ref{thm:upperbound} and the proposition below.

\begin{proposition}\label{prop:a}
Let $Q(Z)$ be a quadratic polynomial and let $M_{\nu} = \rho^{-1}(\nu)$ be a compact regular level set of $\rho$, where $\rho$ is given by
\begin{equation}\label{e:elip}
\rho(Z) = \sum_{k=1}^{n+1} |z^k|^2 + 2\Re Q(Z)
\end{equation}
Then
\begin{equation}
C_{\nu}: = \frac{1}{v(M_{\nu})} \int_{M_{\nu}} |\partial \rho|^{-2} = \frac{1}{\nu}.
\end{equation}
\end{proposition}
\begin{proof} 
We observe that
\begin{align}
\Re \sum_{j=1}^{n+1} z^j \rho^{\jbar}
= 
\sum_{j=1}^{n+1} |z^j|^2 + \Re \sum_{j=1}^{n+1} z^jQ_j
= 
\nu - 2\Re Q + \Re \sum_{j=1}^{n+1} z^jQ_j.
\end{align}
As $Q$ is a quadratic polynomial, we can check that
\(
\sum_{j=1}^{n+1} z ^j Q_j = 2Q.
\) Hence
\begin{equation}
\Re \sum_{j=1}^{n+1} z^j \rho^{\jbar} = \nu 
\quad
\text{on} \ M_{\nu}.
\end{equation}
Therefore,
\begin{align}
\int_{M_{\nu}} |\partial \rho|_{\rho}^{-2}
 = 
 \Re \frac{1}{\nu} \sum_{j=1}^{n+1}\int_{M_{\nu}} \frac{z^j\rho_j}{|\partial \rho|_{\rho}^2} 
 = 
 \Re \frac{1}{\nu} \sum_{j=1}^{n+1} \int_{M_{\nu}} \frac{(z^j+\overline{Q_j})\rho_j}{|\partial \rho|_{\rho}^2} 
 =
 \frac{v(M)}{\nu},
\end{align}
Here, we use 
\begin{equation}
 \int_{M_{\nu}} \frac{\overline{Q_j} \rho_j}{|\partial \rho|_{\rho}^2} 
 =
 \frac{1}{n}\int_{M_{\nu}} \overline{Q_j} \Box_b \bar{z}^j
 = 
 \frac{1}{n}\int_{M_{\nu}} \zbar^j \overline{\Box_b Q_j} 
 =
 0.
\end{equation}
Hence, \(C_{\nu} = \frac{1}{\nu}\).
\end{proof}
\begin{proof}[Proof of Corollary~\ref{cor:ellipsoid}]
From Theorem~\ref{thm:flat} and Proposition~\ref{prop:a}, we have
\begin{equation}
\lambda_1
\leq nC_{\nu} 
=
\lambda_1(\sqrt{\nu}\,\mathbb{S}^{2n+1}).
\end{equation}
Also from Theorem~\ref{thm:flat},  we see that the equality occurs if and only if $M$ is the sphere and hence the proof is complete. However, we 
provide an elementary proof of this last step below. Notice that \begin{equation} Q(Z)=\sum\limits_{k,j=1}^n q_{jk} z_k z_j\end{equation} and \(Q=\big[q_{jk}\big]\) is $n\times n$ symmetric matrix. By an well-known factorization theorem (see \cite[Section~3.5]{Hua}), we can write \(Q=U^t \Lambda U\), 
 where $U$ is a unitary matrix and $\Lambda=\hbox{Diag}(A_1,\cdots, A_{n+1})$ is a diagonal matrix with $A_j\ge 0$. We make
 a holomorphic unitary change of variables $W=U Z$, then
\begin{equation}
 \rho(Z)=\|W\|^2+\Re \sum_{j=1}^{n+1} A_j w_j^2.
 \end{equation}
 Without loss of generality, one may assume that $\rho(Z)=\|Z\|^2+\Re \sum\limits_{j=1}^{n+1}|z_j|^2$. Since $M_{\nu}$ is bounded, it is easy to see that $A_j<1$
 for $1\le j\le n+1$. Notice that on $M_{\nu}$,
\begin{equation}
 c=|\d \rho|_{\rho}^2=\nu+\sum_{j=1}^{n+1} A_j^2 |z_j|^2+\Re\sum_{j=1}^{n+1} A_j z_j^2=2\nu -|Z|^2 +\sum_{j=1}^{n+1} A_j |z_j|^2.
 \end{equation}
 If the equality occurs, then $|\partial\rho|_{\rho}^2$ is a constant along $M_{\nu}$. Restricting $z=\lambda {\bf e}_j\in M_{\nu} $, one has $|z_j|^2 $ must be a constant. This can not be true unless $A_j=0$. This proves $\rho(Z)=\|Z\|^2$ and
 $M_{\nu}$ is a sphere centered at $0$ with radius~$\sqrt{\nu}$.
\end{proof}


\begin{thebibliography}{22}

\bibitem{ADE}
A. Aribi; S. Dragomir; and A. El Soufi,
A lower bound on the spectrum of the sublaplacian. 
Journal of Geometric Analysis, 25.3 (2015) 1492-1519.

\bibitem{BD}
E.~Barletta; S.~Dragomir, On the spectrum of a strictly pseudoconvex {CR} Manifold.
Abhandlungen aus dem Mathematischen Seminar der Universit\"at Hamburg (1997) 67: 33.

\bibitem{BG} R. Beals and P. Greiner, Calculus on Heisenberg manifolds. Ann. of Math. Stud.,
vol. 119, Princeton Univ. Press, New Jersey, 1988.

\bibitem{BdM}
L. Boutet de Monvel, 
Int\'{e}gration des \'{e}quations de Cauchy-Riemann induites formelles,
S\'{e}minaire Goulaoic-Lions-Schwartz, Expose IX (1974-1975)

\bibitem{Burns}
D. Burns.
Global behavior of some tangential Cauchy-Riemann equations.
 Partial Differential Equations and Geometry (Proc. Conf., Park City, Utah), Marcel Dekker, New York 1979.

\bibitem{BE}D. Burns; C. Epstein, Embeddability for Three-dimensional CR
Manifolds, J. Amer. Math. Soc. 4 (1990), 809-840.

\bibitem{CCY}S. Chanillo; H.-L. Chiu; P. Yang, Embeddability for
3-dimensional Cauchy-Riemann manifolds and CR Yamabe invariants. Duke Math. J.
161 (2012), no. 15, 2909--2921.

\bibitem{CW}S.-C. Chang; C.-T. Wu, On the CR Obata Theorem for Kohn
Laplacian in a Closed Pseudohermitian Hypersurface in ${\CC}^{n+1}$.
Preprint, 2012.

\bibitem{CC}
S.-C.~Chang, H.-L.~Chiu,
On the CR analogue of Obata's theorem in a pseudohermitian 3-manifold. Math. Ann. 345 (2009), no. 1, 33–51.

\bibitem{Chiu}
H.-L. Chiu,
The sharp lower bound for the first positive eigenvalue of the sub-Laplacian on a pseudohermitian 3-manifold.
Ann. Global Anal. Geom. 30 (2006), no. 1, 81--96.

\bibitem{Geller1980}
D. Geller,
The Laplacian and the Kohn Laplacian for the sphere. Journal of Differential Geometry. 1980;15(3):417-35.

\bibitem{Gr}
A.~Greenleaf, 
The first eigenvalue of a sub-Laplacian on a pseudohermitian manifold.
Communications in Partial Differential Equations,
10 (1985), no. 2, 191--217.

\bibitem{Hua}
L.~K. Hua,
Harmonic Analysis of Functions of Several Complex Variables in the classical Domains,
Volume 6, Transations of Mathematical Monographs, AMS, Providence, Rhode Island, 1963

\bibitem{K} J. J. Kohn, Boundaries of Complex Manifolds, Proc. Conf. Complex Manifolds (Minneapolis, 1964), Springer-Verlag, New York, 81-94, 1965.

\bibitem{IV}
S.~Ivanov; D. Vassilev.
An Obata type result for the first eigenvalue of the sub-Laplacian on a CR manifold with a divergence-free torsion.
Journal of Geometry 103, no. 3 (2012): 475-504.

\bibitem{L}
J. M. Lee,
The Fefferman metric and pseudohermitian invariants.
Trans. Amer. Math. Soc., 296(1), 411--429.

\bibitem{LL1} S.-Y. Li; H-S Luk, The Sharp lower bound for the 
first positive eigenvalues of sub-Laplacian on
the pseudo-hermitian manifold, {Proc. of AMS}, 132 (2004), 789--798.

\bibitem{LL}
S.-Y. Li; H-S Luk, An explicit formula for the Webster pseudo-Ricci curvature on real hypersurfaces and its application for characterizing balls in $C^n$. Communications in Analysis and Geometry, 14(4), 673--701.

\bibitem{LSW} S.-Y. Li; D. N. Son; X-D. Wang, A New Characterization of the CR Sphere and the sharp eigenvalue estimate for
the Kohn Laplacian.
{Advances in Math.}, 281 (2015), 1285--1305.

\bibitem{LW}S.-Y. Li; X. Wang, An Obata-type Theorem in CR Geometry,
{Journal of Differential Geometry}, 95(2013), no. 3, 483--502.

\bibitem{LiTran}
S.-Y. Li; M.-A. Tran, On the CR-Obata theorem and some extremal problems associated to pseudoscalar curvature on the real ellipsoids in $\mathbb{C}^{n+1}$.
Transactions of the American Mathematical Society 363, no. 8 (2011): 4027-4042.

\bibitem{Ser} J. Serrin, A symmetry problem in potential theory, Arch. Rational Mech. Anal. 43 (1971), 304-318. 

\bibitem{We}
S.~M.~Webster,
Pseudo-Hermitian structures on a real hypersurface.
Journal of Differential Geometry 13, no. 1 (1978): 25-41.

\end{thebibliography}
\end{document}